\documentclass[11pt]{article}

\usepackage{graphicx,amsmath,amssymb,amsthm,color}
\usepackage[all]{xy}
\usepackage[page]{appendix} 
\usepackage{hyperref}

\newtheorem{thm}{Theorem}[section]
\newtheorem{lemma}[thm]{Lemma}

\newtheorem{conj}[thm]{Conjecture}

\newcommand{\QG}{\mathbb{Q}[G]}
\newcommand{\ZG}{\mathbb{Z}[G]}
\newcommand{\ZQ}{\mathbb{Z}[Q_{28}]}
\newcommand{\ZQn}{\mathbb{Z}[Q_{4n}]}

\newcommand{\IQ}{IG^*}
\renewcommand{\vec}[1]{\boldsymbol{#1}}
\input{xy}
\xyoption{all}

\begin{document}

\date{}
\title{An exotic presentation of $Q_{28}$}
\author{W.H. Mannan\footnote{Supported by the Leverhulme Trust}, Tomasz Popiel}


\maketitle

\vspace{-9mm}
\begin{abstract}{We introduce a new family of presentations for the quaternion groups and show that for the quaternion group of order 28, one of these presentations has non-standard second homotopy group.}
\end{abstract}

\bigskip

{\tiny\noindent {\bf MSC classes} Primary: 57M05, 57M20.  Secondary: 20C05, 16S34, 20C10, 55P15, 55Q91, 55N25

}

\section{Introduction}

Since the work of Johnson \cite{John3,John1} and Beyl and Waller \cite{Beyl, Beyl2} in the early 2000's, the hunt has been on to find out if  a finite balanced presentation of a quaternion group $Q_{4n}$ can have non-standard second homotopy group.  This has largely been fuelled by the connection to Wall's famous $D(2)$ problem \cite{John1}.  However until the present work, for each quaternion group $Q_{4n}$, all known presentations have had second homotopy group $IQ_{4n}^*$, the dual of the augmentation ideal, and it was conjectured that anything else would be impossible. 

We show that such a presentation is in fact possible.  That is, the purpose of the present work is to introduce a new family of presentations for $Q_{4n}$, and to show that in the case $n=7$, one (at least) of these presentations has a non-standard second homotopy group:

\bigskip
\noindent{\bf Theorem A.} {\sl We have a presentation for the quaternion group $Q_{28}${\rm:} $$\mathcal{P}'= \langle x,y\,\vert\, y^2=x^7,\quad y^{-1}xyx^{2}=x^3y^{-1}x^2y\rangle,$$  which has a non-standard second homotopy group.  That is, if $X_{\mathcal{P}'}$ is the Cayley complex associated to $\mathcal{P}'$ and $X_{\mathcal{P}}$ is the Cayley complex associated to the standard presentation{\rm:}$$\mathcal{P}= \langle x,y\,\vert\, y^2=x^7,\quad xyx=y\rangle,$$ then $\pi_2(X_{\mathcal{P}'})\not\cong \pi_2(X_{\mathcal{P}})$ as modules over $\ZQ$.}

\bigskip
Note that whilst $\pi_2(X_{\mathcal{P}})$ is known to be generated by a single element over $\ZQ$,  we will show that $\pi_2(X_{\mathcal{P'}})$ is not.  Thus when we say that $\pi_2(X_{\mathcal{P}'})\not\cong \pi_2(X_{\mathcal{P}})$, we mean that this holds with respect to all identifications of the groups presented by $\mathcal{P}$ and $\mathcal{P}'$.  Therefore $X_{\mathcal{P}}$ and $X_{\mathcal{P'}}$ are not homotopy equivalent.

In fact we will show that $X_{\mathcal{P}'}$ has the same second homotopy group as a 3-complex constructed by Beyl and Waller \cite{Beyl}, sometimes referred to in the community as Nancy's Toy.  They state \cite[p.908]{Beyl} that such an $X_{\mathcal{P}'}$, if it exists, will not be homotopy equivalent to the spine of a closed 3--manifold.   We thank J. Nicholson for pointing out that our ${\mathcal{P}'}$ thus resolves the question of whether all finite balanced presentations of closed 3--manifold fundamental groups are homotopy equivalent to such spines.

Another application is given by \cite[Proposition 5.5]{Nich1}, where our construction is used to present non-homotopy equivalent manifolds in dimensions 4 and above, which become diffeomorphic under stabilisation by taking the connected sum with a product of spheres.

\bigskip
Given a presentation $\mathcal{Q}$ for a group $G$, let $X_{\mathcal{Q}}$ denote its Cayley complex.  By the second homotopy group of $\mathcal{Q}$ we refer to the $\mathbb{Z}[G]$ module with underlying abelian group $\pi_2(X_{\mathcal{Q}})$ with natural (right) $G$-action arising intuitively from stretching elements of $\pi_2(X_{\mathcal{Q}})$ back along loops in $G=\pi_1(X_{\mathcal{Q}})$.

It is non-trivial to construct finite presentations $\mathcal{Q}, \mathcal{Q}'$ of the same group $G$, with the same deficiency (number of generators minus number of relators) but with different second homotopy groups.  In particular, the Hurewicz homomorphism identifies $\pi_2(X_{\mathcal{Q}})\cong H_2(\widetilde{X_{\mathcal{Q}}})$, and Schanuel's lemma then implies that: $$\pi_2(X_{\mathcal{Q}})\oplus F \cong \pi_2(X_{\mathcal{Q}'}) \oplus F,$$ for some free finitely generated module $F$ over $\ZG$.

In other words, we require non-cancellation of free modules over $\ZG$.  Note that in the case of finite groups, we have cancellation over $\QG$ for all finitely generated modules.  Thus distinguishing $\pi_2(X_{\mathcal{Q}})$ from $\pi_2(X_{\mathcal{Q}'})$ requires subtle number theoretic considerations.

None the less it has been achieved \cite[\S1.7]{Jens}.  For the trefoil group, Lustig, building on the work of Dunwoody and Berridge produced infinitely many presentations with the same deficiency but pairwise distinct second homotopy groups \cite{Berr,Dunw,Lust}.  For finite groups, homotopically distinct presentations with the same deficiency were found by Metzler for ${C_5}^3$ \cite[p.105]{Metz, Lati1}.  Linnell \cite[Corollary 1.4(iii) and (iv)]{Linn} clarified the situation for second homotopy groups: for $p$ a prime satisfying $p\equiv 1 \mod 4$ there are precisely two homotopy types of presentation for  ${C_p}^3$, but they have isomorphic second homotopy groups.  On the other hand for $p$ a prime satisfying  $p\equiv 1 \mod 6$, we have three homotopically distinct presentations of ${C_p}^4$ and they all have non-isomorphic second homotopy groups (with respect to any identification of the presented groups).

The case of quaternion groups has been the subject of much analysis \cite{John3, John1, Beyl, Beyl2}, largely because of its relation to Wall's $D(2)$ problem.  In 1965 Wall showed  that for $n>2$, if a finite cell complex is cohomologically $n$ dimensional (in the  sense of having no non-trivial cohomology in dimensions above $n$ with respect to any coefficient bundle), then it is in fact homotopy equivalent to an actual $n$ dimensional cell complex \cite{Wall}.  Subsequently it was shown by Swan and Stallings that the only cohomologically 1 dimensional finite cell complexes are disjoint unions of wedges of circles \cite{Stal, Swan1}.  However decades later the case $n=2$ remains a major open problem, known as Wall's $D(2)$ problem.  

A $D(2)$--complex is a finite (connected) 3 dimensional cell complex $Y$, with no cohomology above dimension 2.  To solve the problem one would need to produce a $D(2)$--complex which was not homotopy equivalent to a finite 2-complex (or show that this cannot be done).  Note that without loss of generality such a 2-complex is (the Cayley complex of) a finite presentation of $\pi_1(Y)$.  The existence of such a space $Y$ with a particular fundamental group is equivalent to there being an algebraic 2--complex over the group which is not  geometrically realisable \cite{John5,John1,Mann2,Mann4}.  Using this it has been show that such a space $Y$ cannot have certain fundamental groups, such as cyclic groups, products of the form $C_{\infty} \times C_n$  \cite{Edwa} or dihedral groups \cite{John,John1,Mann1,Shea,Mann6,Hamb}.

On the other hand $D(2)$--complexes have been produced and conjectured to not be homotopy equivalent to any finite presentation of their fundamental group.  Broadly these spaces fall into two categories: 

\begin{enumerate}
\item Those where it is conjectured that there is no finite presentation of their fundamental group with the same Euler characteristic.
\item Those where there are finite presentations of their fundamental group with the same Euler characteristic, but it is conjectured that none of them have the same second homotopy module.
\end{enumerate}

A third less explored option would be to prohibit a finite presentation based on $k$-invariants (see \cite[Chapter 6]{John1}), rather than Euler characteristic or second homotopy group.

Many spaces falling into the first category have been proposed \cite{Brid, Grun, Mann7}.  To actually verify that there is no presentation with sufficiently low Euler characteristic will require a fundamentally novel obstruction.  Ideas from geometric group theory and algebraic geometry \cite{Mann8} have been mooted.  

A quintessential example of a space that fell into the first category had fundamental group a free product of several $C_p \times C_p$ as $p$ ranged over distinct primes.  However it was shown that presentations of these groups with sufficiently low Euler characteristics did indeed exist \cite{Hog}.

As has been mentioned, fundamental groups of spaces falling into the second category require a certain failure in the cancellation of free modules.  Although not necessary, the most prominent examples proposed with finite fundamental group are those where cancellation fails even within the stable class of free modules.  From the  Swan--Jacobinski Theorem \cite[\S15]{John1} we know that such groups must necessarily have a binary polyhedral group as a quotient (see \cite{Nich} for more detailed analysis of which groups this failure of cancellation occurs over).

This makes it natural to look at the binary polyhedral groups themselves. Swan showed that the binary polyhedral groups where cancellation fails in the stable class of free modules are precisely $Q_{4n}$ for $n\geq6$ \cite[Theorem I]{Swan3}.  

Based on this work, spaces were constructed which fell into the second category with fundamental group $Q_{2^k}$ with $k\geq 5$ \cite {John3} and fundamental group $Q_{28}$ \cite{Beyl,Beyl2}.  That is, their second homotopy group was not $IQ_{4n}^*$, and it was conjectured that no finite presentation of $Q_{4n}$ would have a second homotopy group other than $IQ_{4n}^*$.

We prove this conjecture false, by displaying a finite presentation with a second homotopy group different to $IQ_{4n}^*$.  In fact, based on our result Nicholson has shown that there are no solutions to Wall's $D(2)$ problem with fundamental group $Q_{28}$ \cite[Theorem 8.11]{Nich2}.  That is any $D(2)$--complex with fundamental group $Q_{28}$ having minimal Euler characteristic is either homotopy equivalent to $\mathcal{P}$ or $\mathcal{P'}$.  Nicholson has further shown \cite[Theorem 8.10]{Nich2} that our example resolves one direction of Problem D3, from Wall's list \cite{Wall2}. This direction of the problem is a conjecture that any failure of a certain local-global principle results in a solution to the $D(2)$ problem.

It is worth noting that our result means that finite presentations have now been found which defy the relevant conjectures for quintessential examples of spaces which fell into both the first and second category.  It is worth then considering the possibility that finite presentations can always be found, homotopy equivalent to a given $D(2)$--complex.  

Note that a finite presentation of a group (possibly not the fundamental group of the 3--complex) may always be found, so that applying Quillen's plus construction results in a space homotopy equivalent to the 3--complex \cite{Mann5}.  On the other hand a famous result of Bestvina and Brady yields a similar situation where there is no finite presentation of the group at all  \cite{Best}.  

Broadly, the prevailing opinion is that an example from category 1 or 2, will succeed in not being homotopy equivalent to a finite presentation.  The present work is not sufficient to alter that prevailing opinion, but it does draw attention to the possibility.

\bigskip
\noindent {\bf Acknowledgments} We acknowledge that the present work is built on the foundations laid by F.E.A. Johnson, F. Rudolf Beyl, and the late Nancy Waller, who is greatly missed.  We would also like to thank the National Science Foundation for award DMS-0918418 which allowed the first author to meet two of these key players in the field.  Also Johnny Nicholson has made several valuable observations since the first draft of this work was produced, in addition to extending the work in various directions in his own articles.  Finally we thank the referee for diligent scrutiny and helpful suggestions.

\section{The standard presentation}\label{stan}
Let $Q_{4n}$ denote the quaternion group with standard presentation: $$\mathcal{P}= \langle x,y\,\vert\, y^2=x^n,\,\, y=xyx\rangle.$$  
Let $X_{\mathcal{P}}$ denote the Cayley complex of this presentation (where relations $a=b$ are interpreted as relators $a^{-1}b$).  The edges in $X_{\mathcal{P}}$ corresponding  to $x,y$ may be lifted to edges in $\widetilde {X_{\mathcal{P}}}$, represented by generators $\vec{e_1}, \vec{e_2}\in C_1(\widetilde {X_{\mathcal{P}}})$ respectively.  Similarly the two disks in $X_{\mathcal{P}}$ corresponding to the two relations in $\mathcal{P}$ may be lifted to disks in  $\widetilde {X_{\mathcal{P}}}$ , represented by generators 
$\vec{E_1}, \vec{E_2}\in C_2(\widetilde {X_{\mathcal{P}}})$ respectively.

Then $\pi_2(X_{\mathcal{P}})$ is a (right) module over $\mathbb{Z}[\pi_1(X_{\mathcal{P}})]=\ZQn$.  Further $\pi_2(X_{\mathcal{P}})$ may be identified via the Hurewicz isomorphism (as modules over $\ZQn$), with the kernel of the boundary map: $$\partial_2\colon C_2(\widetilde {X_{\mathcal{P}}}) \to C_1(\widetilde {X_{\mathcal{P}}}).$$

We may describe the boundary map $\partial_2$ explicitly as follows:
\begin{align*}
\partial_2\colon \vec{E_1} &\mapsto \vec{e_1}\partial_x(y^{-2}x^n)+\vec{e_2}\partial_y(y^{-2}x^n)&=&\,\,\vec{e_1}\sigma_x-\vec{e_2}(1+y),\\
\partial_2\colon \vec{E_2} &\mapsto \vec{e_1}\partial_x(y^{-1}xyx)+\vec{e_2}\partial_y(y^{-1}xyx)&=& \,\,\vec{e_1}(1+yx)+\vec{e_2}(x-1).
\end{align*}

Here $\partial_x, \partial_y$ denote the free Fox derivative with respect to $x,y$ respectively \cite{Fox} and $\sigma_x$ denotes the group ring element $1+x+x^2+x^3+\cdots+x^{n-1}$.

For proofs of the following see for example \cite[Lemma 4.2]{Beyl} or \cite{John3}.  The module $\pi_2(X_{\mathcal{P}})=$ ker$(\partial_2)$ is generated by: \begin{eqnarray}\vec{u}=\vec{E_1}(x-1)+\vec{E_2}(1-yx).\label{ugen}\end{eqnarray}
Further, the annihilator of $\vec{u}$ is precisely $\Sigma_G \ZQn$, where $\Sigma_G$ denotes the sum of all group elements in $Q_{4n}$.  Letting $\IQ$ denote the module $\ZQn/\Sigma_G\ZQn$, we may conclude that  \begin{eqnarray}\pi_2(X_{\mathcal{P}})= \vec{u}\ZQn \cong \IQ.\label{stanpi2}\end{eqnarray}

\section{The new presentations}\label{pressec}
We now describe a new family of presentations $\mathcal{E}_{n,r}$, where the parameter $r$ is an integer:

$$\mathcal{E}_{n,r}= \langle x,y\,\vert\, y^2=x^n,\quad y^{-1}xyx^{r-1}=x^ry^{-1}x^2y\rangle.$$  

Clearly $Q_{4n}$ is a quotient of the group presented by $\mathcal{E}_{n,r}$, for any $r \in \mathbb{Z}$, as both relations hold for the standard generators $x,y \in Q_{4n}$.  In particular:
\begin{eqnarray}
\label{idenholdsinQ28}  y^{-1}xyx^{r-1}=x^{-1} x^{r-1}= x^{r}x^{-2}=x^ry^{-1}x^2y
\end{eqnarray}

However  $\mathcal{E}_{n,r}$ need not be a presentation for $Q_{4n}$.  If we specialize to $r=3$ though, it is a presentation for $Q_{4n}$, as we shall see.

\begin{lemma} \label{abr=3}
Let $a,b \in G$ for some group $G$ satisfy{\rm:} \begin{eqnarray} ab^2&=&b^3a^2,\label{firstab}\\ba^2&=&a^3b^2. \label{secab}\end{eqnarray} Then $ba=1$.
\end{lemma}

\begin{proof}
Multiplying (\ref{firstab}) through by $a^2$ on the left we get: $$a^2b^3a^2=a^3b^2=ba^2,$$
from (\ref{secab}).  Thus $a^2b^2=1$ so $b^2a^2=1$ and (\ref{secab}) reduces to $b^{-1}=a$.
\end{proof}

\begin{lemma}\label{3isgood}
The presentation $\mathcal{E}_{n,3}$ presents $Q_{4n}$ for all $n \geq 2$.
\end{lemma}

\begin{proof}
In the light of (\ref{idenholdsinQ28}) we know that any relation satisfied by $x,y$ in $\mathcal{E}_{n,3}$, is also satisfied in $\mathcal{P}$.  It remains to show that in the group presented by $\mathcal{E}_{n,3}$, the following identity holds:$$ y=xyx$$
Let $a=y^{-1}xy, b=x$.  From the second relation in $\mathcal{E}_{n,3}$ we have that $ab^2=b^3a^2$.  As $y^2=x^n$ we know that $y^2$ is central and conjugating the second relation in $\mathcal{E}_{n,3}$ by $y$, we get $ba^2=a^3b^2$.  Thus Lemma \ref{abr=3} tells us that $ba=1$.  That is $xy^{-1}xy=1$ and $xyx=y$. 
\end{proof}

Let $\mathcal{P}'$ denote  $\mathcal{E}_{7,3}$.  The remainder of this article will be devoted to showing that $\pi_2(X_{\mathcal{P}'})\not\cong \IQ$, in the case $n=7$.  However we briefly pause to consider other possible presentations $\mathcal{E}_{n,r}$ for $Q_{4n}$.  Computations in Magma \cite{Bosm} suggest that $\mathcal{E}_{n,r}$ is frequently a presentation of $Q_{4n}$.  This has been the case for every value of $n,r$ that we have tried where either $r \not \equiv 2 \mod 3$, or $3 \not\hspace{1mm}\mid n$.  We provide one further result in that direction.

\begin{lemma} \label{abr=2}
Let $a,b \in G$ for some group $G$ satisfy{\rm:} \begin{eqnarray*} ab&=&b^2a^2,\label{firstabr2}\\ba&=&a^2b^2, \label{secabr2}
\\a^n&=&b^n,\end{eqnarray*} where $3\not\hspace{1mm}\mid n$.  Then $ba=1$.
\end{lemma}

\begin{proof}
We have $a^3b=a^2(ab)=a^2b^2a^2=(ba)a^2=ba^3$.  Thus $a^3$ is central and so is $a^n$.  As $3,n$ are coprime we have that $a$ is central.  Thus $ba=1$ follows from either of the first two equations.
\end{proof}

\begin{lemma}
The presentation $\mathcal{E}_{n,2}$ presents $Q_{4n}$ for all $n \geq 2$ with  $3\not\hspace{1mm}\mid n$ .
\end{lemma}

\begin{proof}
Again we need only show that: $$ y=xyx$$ holds in the group with presentation $\mathcal{E}_{n,2}$.  Again let $a=y^{-1}xy, b=x$.  From the second relation in $\mathcal{E}_{n,2}$ we have that $ab=b^2a^2$.  As $y^2=x^n$ we know that $y^2$ is central and conjugating the second relation in $\mathcal{E}_{n,2}$ by $y$, we get $ba=a^2b^2$.  Clearly $a^n=b^n$, so Lemma \ref{abr=2} tells us that $ba=1$.  That is $xy^{-1}xy=1$ and $xyx=y$. 
\end{proof}

\section {Computing $\pi_2(X_{\mathcal{P}'})$}
From now on we fix $n=7$ and we wish to show that $\pi_2(X_{\mathcal{P}'})\not\cong \pi_2(X_{\mathcal{P}})$.  In this section we will describe $\pi_2(X_{\mathcal{P}'})$ as a submodule of $\IQ$ with explicit generators.  Then in \S\ref{Milnor} we will decompose $\pi_2(X_{\mathcal{P}'})$ via Milnor squares, to show that it is indeed not $\pi_2(X_{\mathcal{P}})$.  

First note that multiplying both sides of a relation by the same generator on the same side does not alter the homotopy type of the associated Cayley complex.

Thus replacing the second relation in $\mathcal{P}'$ with any of the following, results in homotopy equivalent Cayley complexes:

\begin{align*}
y^{-1}xyx^2&=x^3y^{-1}x^2y,&{\rm Original \,\, relation}\\
x^{-3}(y^{-1}xyx)x^3 x^{-2}&=y^{-1}x^2y,& {\rm Multiplying\,\, on\,\, left\,\, by\,\,} x^{-3} \\
x^{-3}(y^{-1}xyx)x^3 &=y^{-1}x^2yx^2,&{\rm Multiplying\,\, on\,\, right\,\, by\,\,} x^{2} \\
x^{-3}(y^{-1}xyx)x^3 &=(y^{-1}xyx)(x^{-1}(y^{-1}xyx)x).& {\rm Rebracketing} \\
\end{align*}

Now let $R$ denote the word $y^{-1}xyx$.  The last relation then becomes $x^{-3}Rx^3=R(x^{-1}Rx)$.

Then $\pi_2(X_{\mathcal{P}'})$ may be identified with the kernel of the boundary map $\partial_2'$ associated to the presentation:
$$
 \langle x,y\,\vert\, y^2=x^7,\,\, R(x^{-1}Rx)=x^{-3}Rx^3\rangle.
$$

Let $\vec{F_1}, \vec{F_2}$ denote the generators corresponding to these two relations.  

For a general group presentation containing relators $R_1,\cdots, R_n$, integers $s_1,\cdots,s_n$, words $w_1,\cdots, w_n$ in the generators, and a generator $t$, we have:
$$
\partial_t\left((w_1^{-1}R_1^{s_1}w_1)\cdots(w_n^{-1}R_n^{s_n}w_n)\right)= (\partial_t R_1)s_1w_1 +\cdots+ (\partial_t R_n)s_nw_n
$$
Thus we have:
\begin{eqnarray*}
\partial_x\left((x^{-1}R^{-1}x)R^{-1}(x^{-3}Rx^3)\right)&=&(\partial_x R) (x^3-x-1),\\
\partial_y\left((x^{-1}R^{-1}x)R^{-1}(x^{-3}Rx^3)\right)&=&(\partial_y R) (x^3-x-1).
\end{eqnarray*}

We may describe $\partial_2'$ explicitly:
\begin{align*}
\partial_2'\colon \vec{F_1} &\mapsto \partial_2\vec{E_1},\\
\partial_2'\colon \vec{F_2} &\mapsto \partial_2\vec{E_2}(x^3-x-1).
\end{align*}
Thus given any $\vec{F_1}a+\vec{F_2}b \in {\rm ker}(\partial_2')$, for $a,b \in \ZQ$, we have: \begin{align}\vec{E_1}a+\vec{E_2}(x^3-x-1)b = \vec{u}\gamma,\label{pi2cond}\end{align} for some unique $\gamma \in \IQ$.

We will show that the right annihilator of $x^3-x-1$ is $\{0\}$, so in fact $\gamma$ determines $a,b$.  

\begin{lemma} \label{quarticspan}
In the ring $\mathbb{Z}[x]/(x^{14}-1)$, the ideal generated by $x^3-x-1$ contains $4$.
\end{lemma}

\begin{proof}
Dividing $x^{14}-1$ by $x^3-x-1$ leaves a remainder of \begin{eqnarray*}\alpha_1&=&12x^2+16x+8.\\ \\{\rm Let \hspace{4.8cm}} &\,\,&\\ \alpha_2&=&\alpha_1x - (x^3-x-1)12=16x^2+20x+12,\hspace{4.8cm}\\\alpha_3&=&\alpha_2x-(x^3-x-1)16=20x^2+28x+16.\end{eqnarray*}
Thus $\alpha_1,\alpha_2,\alpha_3$ are divisible by $x^3-x-1$ in the ring $\mathbb{Z}[x]/(x^{14}-1)$.  Finally note $\alpha_3+\alpha_2-\alpha_13=4.$
\end{proof}

Thus we have an element $p\in \mathbb{Z}[x]/(x^{14}-1)\subset \ZQ$ satisfying $p(x^3-x-1)=4$.  If $(x^3-x-1)b=0$ for some $b \in \ZQ$ then $p(x^3-x-1)b=0$ so $4b=0$ and $b=0$.  Thus we can conclude that the right annihilator of $x^3-x-1$ is indeed \{0\}.

\begin{lemma}\label{pi2isker}
We have 
 $$
\pi_2(X_{\mathcal{P}'})\cong\{\gamma \in \IQ |\,\,\exists b\in \ZQ |\,\,(1-yx) \gamma = (x^3-x-1)b\}.$$
In other words, we have that $\pi_2(X_{\mathcal{P}'})$ is the kernel of the homomorphism $$\IQ \to \ZQ/(x^3-x-1)\ZQ,$$ mapping $1\mapsto 1-yx$.
\end{lemma}

\begin{proof}
We have identified $\pi_2(X_{\mathcal{P}'})$ with the kernel of $\partial_2'$, which consists of elements $\vec{F_1}a+\vec{F_2}b$, with $a,b \in \ZQ$ satisfying (\ref{pi2cond}), for some $\gamma \in \IQ$.  From (\ref{stanpi2}) we know that if this condition is satisfied for some $\gamma$, then it is unique.  

Conversely given $\gamma$ satisfying (\ref{pi2cond}), for some $a,b \in \ZQ$, we know that
\begin{align}
a&=(x-1)\gamma,\\
(x^3-x-1)b&=(1-yx) \gamma. \label{gammacond}
\end{align}
As the right annihilator of $(x^3-x-1)$ is $\{0\}$, we know that there is a unique $\vec{F_1}a+\vec{F_2}b\in \pi_2(X_{\mathcal{P}'})$, for which $a,b$ satisfy (\ref{pi2cond}).

Thus $\pi_2(X_{\mathcal{P}'})$ may be identified with the set of $\gamma \in \IQ$, satisfying (\ref{gammacond}), for some $b \in \ZQ$.
\end{proof}

We next seek to better understand the module: $$M= \ZQ/(x^3-x-1)\ZQ.$$

From Lemma \ref{quarticspan} we know that any element of $M$ may be written in the form $a_0+a_1x+a_2x^2+(a_3+a_4x+a_5x^2)y$, with the $a_i\in \{0,1,2,3\}$.  Let $A=\mathbb{Z}_4[x]/(x^3-x-1)$.  Note that in $A$, we have $x(x^2-1)=1$, so $x$ is invertible.

\begin{lemma}
We have a well defined $\ZQ$ module $A \oplus A$ with $\ZQ$ action given by{\rm:} \begin{eqnarray*} (a,b)y&=&(bx^7,a)\\(a,b)x &=&(ax,bx^{-1})\end{eqnarray*} for all $a,b\in A$.
\end{lemma}

\begin{proof}
For any $\mathbb{Z}[x]/(x^{14}-1)$ module $A'$, the above defines a $\ZQ$ action on $A' \oplus A'$ as direct application of $x^7, y^2, xyx, y$ demonstrates that the given action respects the identities $x^7=y^2, xyx=y$.  It thus suffices to show that $x^{14}$ acts trivially on $A$.  We may verify this immediately by recalling from the proof of Lemma \ref{quarticspan} that: $$x^{14}-1=(x^3-x-1)q+4(3x^2+4x+2)$$
for some polynomial $q$ in $x$ with integer coefficients.
\end{proof}

\begin{lemma}
We have an isomorphism of $\ZQ$ modules $M \cong A \oplus A$. \end{lemma}

\begin{proof}
The homomorphism $A \oplus A \to M$ mapping $(a,b)\mapsto a+by$ has inverse $M \to A \oplus A$, mapping $1\mapsto (1,0)$.
\end{proof}

Lemma \ref{pi2isker} identifies $\pi_2(X_{\mathcal{P}'})$ with the kernel of the map $\psi\colon \IQ \to M$, mapping $1 \mapsto 1-yx$.  Let \begin{eqnarray*}\phi_1&=&x^6+x^5-x^4-3x^3-x^2 +x+1,\\ \phi_2&=&2+2x-x^3+x^3y.\end{eqnarray*}

\begin{lemma}
We have $4, \phi_1,\phi_2\in \pi_2(X_{\mathcal{P}'})$.
\end{lemma}

\begin{proof}
Clearly $4\in$ ker $\psi$.  Also $(x^3-x-1)(x^{-3}-x^{-1}-1)$ commutes with $y$ and is divisible by $x^3-x-1$, so it too lies in ker $\psi$. In particular, $\phi_1=-(x^3-x-1)(x^{-3}-x^{-1}-1)x^3$ lies in ker $\psi$.

Finally we note that:
$$(1-yx)\phi_2=(x^3-x-1)(-2+(x^4+x^2+x-1)x^{-4}y).$$
\end{proof}

Our goal in this section is to show that these three elements generate $\pi_2(X_{\mathcal{P}'})$.  To that end we must understand the map $\psi\colon\IQ \to M\cong A \oplus A$.  Firstly, we note the following holds in $A$:

\begin{lemma}\label{x0to5inA}
In $A$ we have\rm{:} $$\begin{array}{lllll}
x^3=x+1,&\hspace{1cm}&x^7=2x^2+2x+1,&\hspace{1cm}&x^{11}=x^2+3x,\\
x^4=x^2+x,&&x^8=2x^2+3x+2,&&x^{12}=3x^2+x+1,\\
x^5=x^2+x+1,&&x^9=3x^2+2,&&x^{13}=x^2+3.\\
x^6=x^2+2x+1,&&x^{10}=x+3,
\end{array}
$$
\end{lemma}

\begin{proof}
To deduce each identity from the preceding one, we need only note that if $x^i=ax^2+bx+c$ in $A$, then $x^{i+1}=bx^2+(a+c)x+a$ in $A$.
\end{proof}

\begin{lemma}\label{psix0to5}
We have{\rm:}\begin{eqnarray*}
\psi(1)&=&(1,3x^2+1),\\
\psi(x)&=&(x,x^2+3x+3),\\
\psi(x^2)&=&(x^2,3x^2+x),\\
\psi(x^3)&=&(x+1,3x+1),\\
\psi(x^4)&=&(x^2+x,x^2+2),\\
\psi(x^5)&=&(x^2+x+1,2x^2+x+2).\\
\end{eqnarray*}
\end{lemma}

\begin{proof}
We note that $(1-yx)x^i \in M$ corresponds to the element $(x^i,-x^{13-i})\in A\oplus A$.  Lemma \ref{x0to5inA} then gives the above expressions.
\end{proof}

\begin{lemma}\label{explicitgen}
The elements $4, \phi_1,\phi_2\in \pi_2(X_{\mathcal{P}'})$ generate $\pi_2(X_{\mathcal{P}'})$ as a right module.
\end{lemma}

\begin{proof}
From any element of $\pi_2(X_{\mathcal{P}'})$, one may subtract appropriate multiples of $\phi_1,\phi_2$, in order to be left with an element $\alpha\in  \pi_2(X_{\mathcal{P}'})$ of the form: $$\alpha=a_0+a_1x+a_2x^2+a_3x^3+a_4x^4+a_5x^5.$$  with the $a_i \in \mathbb{Z}$.
It will suffice to show that $4\vert a_0,a_1,a_2,a_3,a_4,a_5$.  We have $\psi(\alpha)=0$ which by Lemma \ref{psix0to5} is equivalent to
$$
\begin{array}{c}\left(\begin{array}{cccccc} a_0&a_1&a_2&a_3&a_4&a_5\end{array}\right)\\ \\ \\ \\ \\ \end{array} \left(\begin{array}{cccccc}
0&0&1&3&0&1\\
0&1&0&1&3&3\\
1&0&0&3&1&0\\
0&1&1&0&3&1\\
1&1&0&1&0&2\\
1&1&1&2&1&2
\end{array}\right)=\begin{array}{c}\begin{array}{cccccc}\\ \\ \\ \\ \left(0\right.&0&0&0&0&\left.0\right),\end{array}\\ \\ \\ \\ \\ \end{array}
$$ 
working modulo 4.  To deduce that the $a_i\equiv 0 \mod 4$, it suffices to show that the above matrix is invertible, as a matrix over $\mathbb{Z}_4$.  This follows from elementary row or column reduction over $\mathbb{Z}_4$.
\end{proof}

\section{Milnor square decompositions} \label{Milnor}

Lemma \ref{explicitgen} gives us an explicit generating set for $\pi_2(X_{\mathcal{P}'})$ as a submodule of $\IQ$.  In order to show that this is not isomorphic to $\pi_2(X_{\mathcal{P}})$, we will decompose this submodule via a series of Milnor squares (see for example \cite[Section 2]{Beyl}).

Firstly, let $S$ denote the ring $\ZQ/\ZQ(1+y^2)$.  Then \begin{eqnarray*}\pi_2(X_{\mathcal{P}})\otimes_{\ZQ}S &\cong& \IQ/\IQ(1+y^2)\\ &\cong& \ZQ/(\Sigma_G, 1+y^2)\ZQ \cong S,\end{eqnarray*}
as $\Sigma_G=(1+y^2)(1+x+x^2+x^3+x^4+x^5+x^6)(1+y)$.

Then if $N$ denotes the right $S$ module $\pi_2(X_{\mathcal{P'}})\otimes_{\ZQ}S$ we get:

\begin{lemma}\label{Nneedsfree}
If $N$ is not a rank one free module over $S$, then $$\pi_2(X_{\mathcal{P}'}) \not \cong \pi_2(X_{\mathcal{P}})$$ as $\ZQ$ modules.
\end{lemma}

Note that if $N$ has $\mathbb{Z}$ torsion, then it cannot be a rank one free $S$ module and we would have that $\pi_2(X_{\mathcal{P}'}) \not \cong \pi_2(X_{\mathcal{P}})$ as desired.  Hence for the remainder we only need to consider the case where $N$ is $\mathbb{Z}$ torsion free.

\begin{lemma}
The module $N$ is isomorphic to the right ideal of $S$ generated by $\{4, \phi_1,\phi_2\}$.
\end{lemma}

\begin{proof}
The right ideal of $S$ generated by $\{4, \phi_1,\phi_2\}$ is isomorphic to: $$\pi_2(X_{\mathcal{P}'}) /(\pi_2(X_{\mathcal{P}'}) \cap \IQ(1+y^2)).$$

Thus we must show that if $\beta\in\pi_2(X_{\mathcal{P}'})$ and $\beta\in \IQ(1+y^2)$, then $\beta \in \pi_2(X_{\mathcal{P}'})(1+y^2).$  We know that  if $\beta\in\pi_2(X_{\mathcal{P}'})$ and $\beta\in \IQ(1+y^2)$, then $4\beta \in \pi_2(X_{\mathcal{P}'})(1+y^2)$.  Thus $4\beta$ represents 0 in $N$.  As we have that $N$ is $\mathbb{Z}$ torsion free as an assumption, we can conclude that $\beta$ also represents $0$ in $N$.  Thus $\beta\in \pi_2(X_{\mathcal{P}'})(1+y^2)$.
\end{proof}

From now on $N$ will denote the right ideal $(4, \phi_1,\phi_2)S$.  Let $$\sigma_{-x}=1-x+x^2-x^3+x^4-x^5+x^6.$$ We have a Milnor square decomposition of the ring $S$ \cite[\S2, II]{Beyl}:
\begin{eqnarray}\xymatrix{S \ar[d]\ar[r] & S/\sigma_{-x}\ar[d]\\S/(x+1)\ar[r]&S/(x+1,\,\sigma_{-x})}\label{MSqI}
\end{eqnarray}
where the arrows all denote the natural projections.  We have natural identifications: \begin{eqnarray*}  S/(x+1)&=&\mathbb{Z}[y]/(1+y^2),\\S/(x+1,\,\sigma_{-x})&=&\mathbb{Z}_7[y]/(1+y^2).
\end{eqnarray*}

Let $\Lambda=S/\sigma_{-x}$.  Note that $\mathbb{Z}[x]/\sigma_{-x}$ is the cyclotomic ring of degree 7, which embeds in $\mathbb{C}\subset\mathbb{H}$.  This embedding may be extended to embed $\Lambda$ in $\mathbb{H}$. In particular $\Lambda$ contains no zero divisors.  Similarly, the Gaussian integers $\mathbb{Z}[y]/(1+y^2)$ embed in $\mathbb{C}$ and contain no zero divisors.  The ring $\mathbb{Z}_7[y]/(1+y^2)$ is just the field of order 49.  We may rewrite (\ref{MSqI}):
\begin{eqnarray*}\xymatrix{S \ar[d]\ar[r] & \Lambda \ar[d]\\ \mathbb{Z}[y]/(1+y^2)\ar[r]&\mathbb{Z}_7[y]/(1+y^2)}
\end{eqnarray*}

We have a commutative square of modules over the corresponding rings:
\begin{eqnarray}\xymatrix{N \ar[d]^{q_1}\ar[r]^{p_1} & N\otimes\Lambda \ar[d]^{q_2}\\ N\otimes\mathbb{Z}[y]/(1+y^2)\ar[r]^{p_2}&N\otimes\mathbb{Z}_7[y]/(1+y^2)}\label{MSqMod1}
\end{eqnarray}
where again the maps $p_1,p_2,q_1,q_2$ are the natural projections, and each $\otimes$ is over $S$.

\begin{lemma}
We may rewrite the square {\rm (\ref{MSqMod1})} as{\rm:}
\begin{eqnarray}\xymatrix{N \ar[d]^{q_1}\ar[r]^{p_1} & N/S\sigma_{-x}\ar[d]^{q_2}\\ \mathbb{Z}[y]/(1+y^2)\ar[r]^{p_2}&\mathbb{Z}_7[y]/(1+y^2)}\label{MSq}
\end{eqnarray}
where $p_1$ is the natural projection, $p_2$ is reduction modulo 7, and $q_1,q_2$ are restrictions of the ring homomorphisms{\rm :} $$S \to \mathbb{Z}[y]/(1+y^2),\qquad\qquad \Lambda \to \mathbb{Z}_7[y]/(1+y^2),$$
respectively, both mapping $x \mapsto -1, y \mapsto y$.
\end{lemma}

\begin{proof}
In $S$ we have $\sigma_{-x}=\phi_1\sigma_{-x}$, so we have $N\sigma_{-x}=S\sigma_{-x}$.  Thus $$N\otimes\Lambda\cong N/N\sigma_{-x}\cong N/S\sigma_{-x},$$ and the map $p_1$ is the natural projection.

We have: \begin{eqnarray}\phi_1=1+(x^5-x^3-2x^2+x)(x+1).\label{phi1is1}\end{eqnarray}
Given $w \in N \cap S(x+1)$ we have $w=a(x+1)+by(x+1)$ for some polynomial expressions $a,b$ in $x$ over $\mathbb{Z}$.  Note that $y\phi_1=\phi_1x^{-6}y$.  Thus we have \begin{eqnarray*}w\phi_1&=&(a+by)\phi_1(x+1)\\&=&(a\phi_1+b\phi_1 x^{-6}y)(x+1)\\&=&\phi_1(a+bx^{-6}y)(x+1).
\end{eqnarray*}
Thus multiplying (\ref{phi1is1}) on the left by $w$ and rearranging gives:
\begin{eqnarray*}w&=&\phi_1(a+bx^{-6}y)(x+1) - w(x^5-x^3-2x^2+x)(x+1).\end{eqnarray*}  In particular $w \in N(x+1).$  Thus we have $N \cap S(x+1) =N(x+1)$.  We conclude: \begin{eqnarray*}N \otimes \mathbb{Z}[y]/(1+y^2)&\cong& N/N(x+1) \\&\cong& N/(N \cap S(x+1)) \subseteq \mathbb{Z}[y]/(1+y^2),\end{eqnarray*} and $q_1$ is the desired restriction.  However from (\ref{phi1is1}) we know $q_1(\phi_1)=1$, so in fact we may identify $N \otimes \mathbb{Z}[y]/(1+y^2)\cong \mathbb{Z}[y]/(1+y^2)$.

Finally note that $q_1(\sigma_{-x})=7$, so $p_2$ is reduction modulo 7, and the square commutes, so $q_2$ must also be the desired restriction.
\end{proof}

We have $N/S\sigma_{-x}\subset S/S\sigma_{-x}\cong \Lambda$.  Thus we may identify $N/S\sigma_{-x}$ with the right ideal, $I=(4,\phi_1,\phi_2)\Lambda$.

\begin{lemma}
The right ideal  $I$ is freely generated over $\Lambda$ by the element $1+yx$.
\end{lemma}

\begin{proof}
In $\Lambda$ we have: $$\phi_1=\phi_1+\sigma_{-x}=2(x^3-1)^2.$$However we also know that:$$(x^3-1)(-x^3+x^2-x)=-x^6+x^5-x^4+x^3-x^2+x=1,$$ so $x^3-1$ is a unit and $2\in I$.

Thus $-x^3+x^3y=\phi_2-2(1+x)\in I$.  Multiplying (on the right) by $x^4$ gives us that $1+yx \in I$.

We next show that $1+yx$ divides (on the left) $4, \phi_1,\phi_2$.  Note first that $(1+yx)(1-yx)=2$, so $2\in (1+yx)\Lambda$.  Thus $4\in (1+yx)\Lambda$ and $\phi_1=2(x^3-1)^2\in (1+yx)\Lambda$.

Finally note that: $$\phi_2=2(1+x)-(1+yx)x^3\in (1+yx)\Lambda .$$
We conclude that $1+yx$ generates the ideal $I$.  Further, as $\Lambda$ contains no zero divisors, we know that $1+yx$ must generate $I$ freely.
\end{proof}

Suppose now that $N$ is free of rank one.  Then it must be freely generated by some element $v\in N$.  Then $p_1(v),q_1(v)$ must freely generate $I,\,\, \mathbb{Z}[y]/(1+y^2)$ respectively.  That is: \begin{eqnarray*}p_1(v)&=&(1+yx)\mu_1,\\q_1(v)&=&\mu_2,\end{eqnarray*} for units $\mu_1 \in \Lambda^*, \mu_2 \in (\mathbb{Z}[y]/(1+y^2))^*$.  By commutativity of (\ref{MSq}) we have:\begin{eqnarray} p_2(\mu_2)=q_2((1+yx)\mu_1).\label{Mcomm}\end{eqnarray}  Let \begin{eqnarray*}\hat{p_2}\colon (\mathbb{Z}[y]/(1+y^2))^* &\to& (\mathbb{Z}_7[y]/(1+y^2))^*,\\ \hat{q_2}\colon \Lambda^* &\to& (\mathbb{Z}_7[y]/(1+y^2))^*,\end{eqnarray*} denote the maps on units induced by the natural projections.  Then from (\ref{Mcomm}) we get: \begin{eqnarray}\hat{p_2}(\mu_2)=(1-y)\hat{q_2}(\mu_1).\label{Nfreecontr}\end{eqnarray}

The proof of the following lemma is usually deferred to the proof of \cite[Lemma 10.13]{Swan3} in Swan's long paper (see for example \cite[Theorem 3.2]{Beyl}).  However, in order to keep our proof of Theorem A self-contained, we give a more elementary proof, based on a number theoretic result proved in Appendix \ref{elem}.

\begin{lemma}
Let $H$ denote the subgroup of the abelian group $$(\mathbb{Z}_7[y]/(1+y^2))^*,$$ generated by the images of $\hat{p_2}, \hat{q_2}$.  Then $H$ is generated by $3,y$ and has cosets $H, (1+2y)H, (-3+4y)H, (1+4y)H$. \label{longpaper}
\end{lemma}

\begin{proof}

Let $L$ denote the multiplicative span of $3,y$, so $$L=\{1,2,3,4,5,6,y,2y,3y,4y,5y,6y\}.$$  We will show that $H=L$ and the cosets of $H$ follow. 

Applying $\hat{q_2}$ to units in $\Lambda$ of the form: $$1+(-x)+ \cdots+(-x)^r \quad{\rm or}\quad(1+(-x)+ \cdots+ (-x)^r)y,$$ for $r=0,1,2,3,4,5$, we obtain all the elements of $L$.

The units in the Gaussian integers $\mathbb{Z}[y]/(1+y^2)$ are just $\{1,y,-1, -y\}$, as for $a+by$ to be a unit with $a,b\in\mathbb{Z}$, we must have $a^2+b^2=1$.  Thus the image of $\hat{p_2}$ lies in the (multiplicative) span of $y$.

It remains to show that the image of $\hat{q_2}$ is contained in $L$.  First recall that $\Lambda$ embeds in the quaternions $\mathbb{H}$.  Specifically, if $\zeta=e^{\frac{2\pi i}7}$ then we have an embedding sending: $$x\mapsto -\zeta,\qquad y\mapsto j.$$
We thus identify $\Lambda=\mathbb{Z}[\zeta,j]\subseteq\mathbb{H}$.  The quotient map  $\Lambda\to S/(x+1,\,\sigma_{-x})$ has kernel generated by $x+1$ which is identified with $1-\zeta$.

Any unit in $\Lambda$ has the form $u$ or $vj$ with $u,v\in \mathbb{Z}[\zeta]$ (see lemma \ref{Quat} in Appendix \ref{elem}).  Thus the unit will map to an element of $L$, under  $\hat{q_2}$.
\end{proof}

\begin{lemma}\label{Nnotfree}
The module $N$ is not free.
\end{lemma}

\begin{proof}
If $N$ were free then by (\ref{Nfreecontr}) we would have $1-y \in H$.  However $-3y(1-y)=(-3+4y)$, so $1-y\in (-3+4y)H$.
\end{proof}

Combining lemmas \ref{3isgood}, \ref{Nneedsfree}, \ref{Nnotfree} we deduce:

\bigskip
\noindent{\bf Theorem A.} {\sl We have a presentation for the quaternion group $Q_{28}${\rm:} $$\mathcal{P}'= \langle x,y\,\vert\, y^2=x^7,\quad y^{-1}xyx^{2}=x^3y^{-1}x^2y\rangle,$$  which has a non-standard second homotopy group.  That is, if $X_{\mathcal{P}'}$ is the Cayley complex associated to $\mathcal{P}'$ and $X_{\mathcal{P}}$ is the Cayley complex associated to the standard presentation{\rm:}$$\mathcal{P}= \langle x,y\,\vert\, y^2=x^7,\quad xyx=y\rangle,$$ then $\pi_2(X_{\mathcal{P}'})\not\cong \pi_2(X_{\mathcal{P}})$ as modules over $\ZQ$.}

\bigskip
The fact that our procedure resulted in the coset $(-3+4y)H$ actually tells us (see proof of \cite[Theorem 3.2]{Beyl}) that our presentation $\mathcal{P}'$ has the same second homotopy group as the algebraic 2--complex constructed in \cite{Beyl}: the so called Nancy's Toy \cite[\S1.9.4]{Jens}.  This is no surprise given that $N$ is non-free, as from \cite[pp. 110--111]{Swan3} we know that $(-3+4y)H$ is the only coset corresponding to a non-free stably free module and $N$ had to be stably free, as the Hurewicz isomorphism theorem and Schanuel's lemma combine to imply $\pi_2(X_{\mathcal{P}'})$ and $\pi_2(X_{\mathcal{P}})$ are stably equivalent.

\section{The $D(2)$-property for $Q_{4n}$}\label{D2Q4n}

Thus we have shown that it is possible for a finite balanced presentation of $Q_{28}$ to have a non-standard second homotopy group.  Let $Y$ be a $D(2)$--complex, with $\pi_1(Y)=Q_{4n}$ for some $n\geq2$. In particular we have shown that $Y$ having a non-standard second homotopy group is not sufficient for it to solve Wall's $D(2)$ problem.

 One might ask if every such $Y$ of minimal Euler characteristic is homotopy equivalent to $\mathcal{E}_{n,r}$ for some $r$.  Nicholson has answered this question in the negative.  In the discussion proceeding \cite[Theorem B]{Nich4} he notes that whilst the number of homotopically distinct presentations in our family grows linearly in $n$, the number of minimal $Y$ as above grows exponentially.

Nonetheless he does show that our presentations are enough to verify the $D(2)$ property for $Q_{28}$ \cite[Theorem 8.11]{Nich2}.

It remains possible that: 

\begin{conj} \label{Qconj} Every $D(2)$--complex $Y$ with $\pi_1(Y)=Q_{4n}$ and $\chi(Y)=1$, is homotopy equivalent to a presentation of $Q_{4n}$ of the form{\rm:}
$$\mathcal{Q} =\langle x,y\,\vert\, y^2=x^n,\quad {\rm Eq}(y^{-1}xy, x)\rangle,$$ where ${\rm Eq}(a,b)$ is an equation implied by $ab=1$, equating words in $a,b$.\end{conj}

 As a starting point for proving this, we would require a more systematic way of computing generators of $\pi_2(\widetilde {X_{\mathcal{Q}}})$.  We are grateful to the referee for suggesting the following approach:

Let $\vec{G_1},\vec{G_2}$ be generators of $C_2(\widetilde {X_{\mathcal{Q}}})$ corresponding to the relations of $\mathcal{Q}$, respectively.  Let $\partial''_2\colon C_2(\widetilde {X_{\mathcal{Q}}})\to C_1(\widetilde {X_{\mathcal{Q}}})$ be the boundary map.  Recall from section \ref{stan} the generators $\vec{E_1},\vec{E_2}$ for $C_2(\widetilde {X_{\mathcal{P}}})$, where $\mathcal{P}$ was the standard presentation for $Q_{4n}$.    As the second relation of $\mathcal{Q}$ is implied by the second relation of $\mathcal{P}$ and $\mathcal{Q}$ presents $Q_{4n}$ we have: \begin{eqnarray*}\partial_2(\vec{E_1})&=&\partial''_2(\vec{G_1})\\\partial''_2(\vec{G_2})&=&\partial_2(\vec{E_2})\lambda,\\ \partial_2(\vec{E_2})&=&\partial''_2(\vec{G_1})\mu_1+
\partial''_2(\vec{G_2})\mu_2,
\end{eqnarray*} for some $\lambda,\mu_1,\mu_2\in\ZQn$.  

Thus we have a commutative diagram:
$$\xymatrix{ 
0\ar[r]&\pi_2(\widetilde {X_{\mathcal{Q}}})\ar@<-1ex>[d]\ar@{^{(}->}[r] & C_2(\widetilde {X_{\mathcal{Q}}})\ar@<-1ex>[d]_f \ar[r]^{\partial_2''}&C_1(\widetilde {X_{\mathcal{Q}}})\ar@{=}[d]\\
0\ar[r]&\pi_2(\widetilde {X_{\mathcal{P}}}) \ar@<-1ex>[u]\ar@{^{(}->}[r] &C_2(\widetilde {X_{\mathcal{P}}})\ar@<-1ex>[u]_g\ar[r]^{\partial_2}&C_1(\widetilde {X_{\mathcal{P}}})}
$$
where $f,g$ are given by:$$f\colon \begin{array}{l}
\vec{G_1}\mapsto \vec{E_1},\\ \vec{G_2}\mapsto \vec{E_2}\lambda,
\end{array}\qquad\qquad
g\colon \begin{array}{l}
\vec{E_1}\mapsto \vec{G_1},\\ \vec{E_2}\mapsto \vec{G_1}\mu_1+\vec{G_2}\mu_2.\end{array}
$$

The following may then be useful for computing $\pi_2(\widetilde {X_{\mathcal{Q}}})$ in various cases, as a starting point to prove conjecture \ref{Qconj}.  

\begin{lemma}
Regarding $\pi_2(\widetilde {X_{\mathcal{Q}}})$ as a submodule of $C_2(\widetilde {X_{\mathcal{Q}}})$ in the natural way, it is generated by the elements:$$\vec{G_1}((x-1)+\mu_1(1-yx))+\vec{G_2}\mu_2(1-yx),\quad\quad -\vec{G_1}\mu_1\lambda+\vec{G_2}(1-\mu_2\lambda).$$
\end{lemma}

\begin{proof}
Suppose $\vec{G_1}\alpha_1+\vec{G_2}\alpha_2\in \pi_2(\widetilde {X_{\mathcal{Q}}})$.  We have: $$
\vec{G_1}\alpha_1+\vec{G_2}\alpha_2=gf(\vec{G_1}\alpha_1+\vec{G_2}\alpha_2)+((1-gf)G_2)\alpha_2.$$ Here
$f(\vec{G_1}\alpha_1+\vec{G_2}\alpha_2)\in \pi_2(\widetilde {X_{\mathcal{P}}})$ so $$ f(\vec{G_1}\alpha_1+\vec{G_2}\alpha_2)=(\vec{E_1}(x-1)+\vec{E_2}(1-yx))\gamma,$$ for some $\gamma\in\ZQn$, by (\ref{ugen}).

Thus $\pi_2(\widetilde {X_{\mathcal{Q}}})$ is generated by $g(\vec{E_1}(x-1)+\vec{E_2}(1-yx)), (1-gf)G_2$ over $\ZQn$ .  We conclude by noting:
\begin{eqnarray*}
g(\vec{E_1}(x-1)+\vec{E_2}(1-yx))&=& \vec{G_1}((x-1)+\mu_1(1-yx))+\vec{G_2}\mu_2(1-yx),
\\
(1-gf)(G_2)&=&-\vec{G_1}\mu_1\lambda+\vec{G_2}(1-\mu_2\lambda).
\end{eqnarray*} \vspace{-13mm}{\,}

\end{proof}

However at present $Q_{28}$ remains the only quaternionic group with chain homotopically distinct minimal algebraic $2$-complexes for which the $D(2)$ property has been verified.  Recent progress for $Q_{24}$ (see discussion proceeding \cite[Lemma 8.3]{Nich3}) shows that the two non-standard algebraic 2--complexes would be realised by a single exotic presentation (with different identifications of the presented group with $Q_{24}$).  However even in this case it is not known if such a presentation exists.

\newpage
\begin{appendices} 
\section{Units in rings of integers} \label{elem}

\begin{lemma}\label{Gal}
Let $K\subset \mathbb{C}$ be a finite degree Galois extension of $\mathbb{Q}$. Suppose that complex conjugation is central in Gal$(K/\mathbb{Q})$.  Then for any $u,v \in \mathcal{O}_K$, if $u\bar{u}+v\bar{v}$ is a unit in $\mathcal{O}_K$, then $u=0$ or $v=0$.
\end{lemma}

\begin{proof} Suppose $u\bar{u}+v\bar{v}$ is a unit, but $u\neq0, \, v\neq 0$. Let $G={\text Gal}(K/\mathbb{Q})$.  We have \begin{eqnarray}\prod_{g \in G} g(u\bar{u}+v\bar{v})=1,\label{equalsone}\end{eqnarray}
as the product on the left  must be an integer (as it is invariant under ${\text Gal}(K/\mathbb{Q})$), a unit (as it is a product of units) and positive (as it is a product of positive real numbers, using the centrality of complex conjugation).    

However, expanding the product, we get a sum of positive real numbers (again by the centrality of complex conjugation), including a pair of integers: \begin{eqnarray*}\prod_{g \in G} g(u\bar{u}+v\bar{v})\geq \prod_{g \in G} g(u\bar{u})+ \prod_{g \in G} g(v\bar{v})\geq 2,\end{eqnarray*}
contradicting (\ref{equalsone}).
\end{proof}

\begin{lemma}\label{Quat}
Let $K\subset \mathbb{C}\subset\mathbb{H}$ be a finite degree Galois extension of $\mathbb{Q}$. Suppose that complex conjugation is central in Gal$(K/\mathbb{Q})$.  Then for any $u,v \in \mathcal{O}_K$, if $u+vj$ is a unit in $\mathcal{O}_K[j]$, then $u=0$ or $v=0$.
\end{lemma}

\begin{proof}
If $u+vj$ is a unit, then so is its quaternion conjugate $\bar{u}-vj$.  Thus their product $u\bar{u}+v\bar{v}$  is a unit in $\mathcal{O}_K$.  From lemma \ref{Gal} we then know that $u=0$ or $v=0$.
\end{proof}

It is worth noting that the centrality condition on complex conjugation is not redundant here, as we have for example the unit:  $$\sqrt{2-\sqrt2}+ \left(\sqrt{\sqrt2-1}\right)j \in \mathcal{O}_K[j],$$
where $K=\mathbb{Q}\left[\sqrt{2\pm\sqrt2},\sqrt{\pm\sqrt2-1}\right]$.  We thank Cam Mcleman for his \href{https://mathoverflow.net/questions/38680/can-an-algebraic-number-on-the-unit-circle-have-a-conjugate-with-absolute-value/38683#38683}{post} regarding this number.

\end{appendices}

\newpage

\noindent W.H. Mannan,  \hfill {\it email}: w.mannan@qmul.ac.uk

\noindent Queen Mary University of London, 

\noindent School of Mathematical Sciences, 

\noindent Mile End Road,

\noindent London E1 4NS.

\bigskip\bigskip
\noindent Tomasz Popiel, \hfill {\it email}: tomasz.popiel@auckland.ac.nz

\noindent The University of Auckland,

\noindent Department of Mathematics,

\noindent Auckland 1142.

\end{document}